\documentclass[11pt,twoside,a4paper]{amsart}

\usepackage{amsmath,amsthm,amsfonts,amssymb}
\usepackage{amsthm}
\usepackage{graphicx}
\usepackage{hyperref}
\usepackage{color}
\usepackage[utf8]{inputenc}

\makeatletter
\newtheorem*{rep@theorem}{\rep@title}
\newcommand{\newreptheorem}[2]{%
\newenvironment{rep#1}[1]{%
 \def\rep@title{#2 \ref{##1}}%
 \begin{rep@theorem}}%
 {\end{rep@theorem}}}
\makeatother

\numberwithin{equation}{section}
\theoremstyle{plain}
\newtheorem{theorem}{Theorem}[section]
\newreptheorem{theorem}{Theorem}

\newtheorem{lemma}[theorem]{Lemma}
\newtheorem{corollary}[theorem]{Corollary}
\newreptheorem{corollary}{Corollary}
\newtheorem{conjecture}{Conjecture}[section]
\theoremstyle{definition}
\newtheorem{definition}{Definition}[section]
\newtheorem{remark}[theorem]{Remark}
\newtheorem{notation}{Notation}[section]
\newtheorem{example}[theorem]{Example}

\usepackage{tikz-cd}

\begin{document}
\title{On the stability of kernel bundles over chain-like curves}
\address{Department of Mathematics, Amrita School of Engineering, Bangalore, Amrita Vishwa Vidyapeetham, India}
\email{chuchabn@gmail.com}
\author{Suhas B N, Susobhan Mazumdar and Amit Kumar Singh}
\address{Department of Mathematics, Indian Institute of Technology-Madras, Chennai, India }
\email{susobhan.mazumdar@gmail.com}
\address{The Institute of Mathematical Sciences, C.I.T. Campus, Taramani, Chennai- 600113, India}
\email{amitsingh@imsc.res.in / amitks.math@gmail.com}
\maketitle
\abstract 
Let $C$ be a chain-like curve having $n$ smooth components and $n-1$ nodes, where $n \geq 2$. Let $E$ be a vector bundle on $C$ and $V \subseteq H^0(E)$ be a linear subspace generating $E$. We investigate the (semi)stability of the \textit{kernel bundle} $M_{E,V}$ associated to $(E,V)$. 
\endabstract

\section{Introduction}
Let $C$ be a reduced projective curve over the field $\mathbb{C}$ of complex numbers. Let $E$ be a vector bundle of rank $r$ on $C$. Suppose $V$ is a linear subspace of $H^0(E)$ of dimension $k > r$ such that $V$ generates $E$. Such a pair $(E, V)$ is called a generated pair on $C$. The kernel bundle $M_{E, V}$ on $C$ of the generated pair $(E, V)$ is defined by the following exact sequence :
 \begin{equation}\label{defining sequence for the kernel bundle}
0 \rightarrow M_{E,V} \rightarrow V \otimes \mathcal{O}_C \rightarrow E \rightarrow 0.
\end{equation}
\par 
Also, the generated pair $(E, V)$ gives a morphism 
\[
\varphi_{E, V} : C \rightarrow Gr(r, V) := G,
\]
with  $ E \cong \varphi_{E,V}^*(Q)$, where $Q$ is the quotient bundle on $G$ defined by the exact sequence
\begin{equation}
0 \rightarrow S \rightarrow V \otimes \mathcal{O}_G \rightarrow Q \rightarrow 0 ,
\end{equation}
and $S$ is the universal bundle on $G$. Therefore, $M_{E, V}$ can be observed as the pull back of the universal bundle $S$ on $G$ via the morphism $\varphi_{E, V}$. In particular, for the case of a line bundle $L$, we have $\varphi^*_{L,V} (\Omega_{\mathbb{P}}(1)) \cong M_{L, V}$, where $\mathbb{P}$ and $\Omega_{\mathbb{P}}(1)$ are respectively the projective space $\mathbb{P}(V^*)$ and its twisted cotangent bundle. When $V = H^0(E)$, we denote the kernel bundle by $M_E$ instead of $M_{E, V}$. \par 
The question of (semi)stability of $M_{E, V}$ is a long standing problem. One of  the reasons for studying this question lies in its application to the syzygy problems \cite{Ein92, Kapil-thesis}.  It also has a deep connection with the higher rank Brill-Noether Theory \cite{Beauville2006}.
\\ 
Over a smooth projective curve $C$ of genus $g \geq 2$, it was proved by Ein and Lazarsfeld that $M_{L}$ is (semi)stable whenever $L$ is a line bundle on $C$ with $\deg L  (\geqslant) > 2g$ \cite[Proposition 3.1]{Ein92}. This result was generalised for the case of a (semi)stable vector bundle by Butler. More precisely, he proved that $M_{E}$ is (semi)stable if $\deg E (\geqslant) > 2gr$, where $E$ is a rank $r$ (semi)stable vector bundle on a smooth curve $C$ of genus $g \geqslant 2$ \cite[Theorem 1.2]{Butler94}. Much later, the bound on the degree was improved for the case of a line bundle by Camere \cite[Theorem 1.3]{Camere2008}. Moreover, the bound obtained by her is the best known bound so far.  
\par 
The (semi)stability of $M_{E,V}$ has been studied in connection with theta-divisor \cite{Beauville03, Beauville2006, Mistretta08, Popa99} as well. But still not much is known when $V \neq H^0(E)$. 
\par 
In \cite{Butler-conj}, Butler made a conjecture on the semistability of $M_{E, V}$ for a semistable vector bundle $E$ on $C$. One version of the conjecture is as follows:
 \begin{conjecture}
Let $C$ be a general smooth curve of genus $g \geqslant 3$. Then, for a general choice of a generated pair $(E, V)$ on $C$, where $E$ is a semistable vector bundle, the kernel bundle $M_{E, V}$ is semistable.
 \end{conjecture}
 A lot of work has been done to address this conjecture (see for example,\cite{Hein98, Usha08, Brambila-Newstead2019}). Nevertheless, it is yet to reach complete solution. In 2015, Bhosle, Brambila-Paz and Newstead proved the conjecture for a line bundle $L$ by using the wall crossing formulae for coherent systems of a curve \cite{Usha15}. 
 Though there has been a considerable amount of literature on the (semi)stability of the kernel bundle over a smooth projective curve, not much is known for the case of a singular curve. \\
 
 In 2019, Brivio and Favale \cite{Brivio-Favale2020} studied the (semi)stability of $M_{E, V}$ on a reducible curve with a node. Surprisingly, they were able to prove results which go in the opposite direction with respect to what is known in the smooth curve case. More precisely, suppose $C$ is a reducible nodal curve with two smooth components $C_1$ and $C_2$ intersecting at a node $p$, $(E,V)$ is a generated pair on $C$ and $E_j$ is the restriction of $E$ to $C_j$ for $j = 1,2$. Then Brivio and Favale proved that whenever $E_j$ is semistable on $C_j$ and $V \cap H^0(E_j(-p)) \neq 0$ for each $j$, then the kernel bundle $M_{E,V}$ is strongly unstable (\cite[Theorem 2.4]{Brivio-Favale2020}). Further, they also provided conditions under which $M_{E,V}$ becomes $w$-semistable for a suitable choice of polarization $w$ (\cite[Theorem 3.2]{Brivio-Favale2020}).

 Motivated by their results, in this article, we consider the same question for the case of a \textit{chain-like curve} $C$ having $n$ smooth components $C_1,\dots,C_n$ and $n-1$ nodes $p_1,\dots,p_{n-1}$, where $n \geq 2$ (see section $2$ for the precise definition of a chain-like curve). We first prove a more general result about the $w$-semistability of any subbundle of a trivial bundle over $C$. More precisely, we prove the following theorem:

 \begin{reptheorem}{general result on stability}
  Let $C$ be a chain-like curve. Suppose $\mathcal{V}$ is a nonzero finite dimensional vector space over $\mathbb{C}$ and $\mathcal{M}$ is a subbundle of the trivial bundle $\mathcal{V} \otimes \mathcal{O}_C$ on $C$ such that $\text{rk}(\mathcal{M}) = m$, $\chi(\mathcal{M}) < 0$ and the quotient $\mathcal{E} : = \frac{\mathcal{V} \otimes \mathcal{O}_C}{\mathcal{M}}$ is locally free. Suppose $\mathcal{M}_{|_{C_j}}$ is semistable for each $j$. Then there exists a polarization $w$ such that $\mathcal{M}$ is $w$-semistable. Further, if $\mathcal{M}_{|_{C_j}}$ is stable for some $j$, then $\mathcal{M}$ is $w$-stable.
 \end{reptheorem}
 Since the kernel bundle $M_{E,V}$ associated to a generated pair $(E,V)$ is a subbundle of $V \otimes \mathcal{O}_C$, as a corollary to the above theorem, we obtain the following result:
 
 \begin{repcorollary}{stability of the kernel bundle}
 Let $C$ be a chain-like curve and $(E,V)$ be a generated pair on $C$. If for each $j$, ${M_{E,V}}_{\vert_{C_j}}$ is semistable, then there exists a polarization $w$ such that $M_{E,V}$ is $w$-semistable.  
\end{repcorollary}

Here, we would like to emphasize that 
if the number of components in $C$ is greater than $2$, the inequalities involving the polarization and Euler characteristics (the inequalities (\ref{polarization inequalities}) to be precise) will be slightly more complicated than the case when $n=2$. So the proof for existence of such a polarization $w$ which makes $M_{E,V}$ $w$-semistable, is a little more involved (see Lemma \ref{lemma on existence of polarization}). \\

We then provide some sufficient conditions under which $M_{E,V}$ and $M_{E,V} \otimes L$ become strongly unstable, where $L$ is any line bundle on $C$. More precisely, we prove the following:
\begin{reptheorem}{a theorem on strongly unstable kernel bundles}
 Let $(E,V)$ be a generated pair on a chain-like curve $C$. Then $M_{E,V}$ is strongly unstable, if any of the following conditions hold:
 \begin{enumerate}
     \item[(a)] $V \cap H^0(E_1(-p_{1})) \neq 0$ (or $V \cap H^0(E_n(-p_{n-1})) \neq 0$), $E_1$ (resp. $E_n$) is semistable and $(k-r) < d_1$ (resp. $(k-r) < d_n$).
     \item[(b)] there exists $j \in \{2,\dots,n-1\}$ such that $V \cap H^0(E_j(-p_{j-1}-p_j)) \neq 0$, $E_j$ is semistable and $(k-r) < \frac{d_j}{2}$,
 \end{enumerate}
 where $d_j := \deg(E_j)$.
\end{reptheorem}

\begin{reptheorem}{theorem for strongly unstable for any L}
Let $(E, V)$ be a generated pair on a chain-like curve $C$ with $ \frac{d}{k-r} > (n-1)$. Suppose for each $j$, $\rho_j$ denotes the natural map $H^0(E) \rightarrow H^0(E_j)$ and $\text{Ker}({\rho_j}_{|_V}) \neq 0$. Then $M_{E, V} \otimes L$ is strongly unstable, for any line bundle $L$ on $C$. In particular, $M_{E,V}$ is strongly unstable.
\end{reptheorem}
As a corollary, we show that when $n=2$, $E_j$ is semistable for $j=1,2$ and $L = \mathcal{O}_C$, the last theorem stated above is nothing but \cite[Theorem 2.4]{Brivio-Favale2020}.

\section{Preliminaries}
Let $n \geq 2$ be a positive integer. Let $C$ be a projective reducible nodal curve over the field $\mathbb{C}$ of complex numbers having $n$ smooth irreducible components $C_i$ of genus $g_i \geq 2$ and $n-1$ nodes $p_i$ such that $C_i \cap C_j = \emptyset$ whenever $|i-j| > 1$ and $C_i \cap C_{i+1} = \{p_i\}$ for $i = 1,\dots,n-1$. We call such a curve a \textit{chain-like} curve.

On such a curve $C$, 
we have the following exact sequence:

 \begin{equation}\label{canonical exact sequence}
    0 \rightarrow \mathcal{O}_C \rightarrow \bigoplus_{j=1}^n\mathcal{O}_{C_j} \rightarrow \mathcal{T} \rightarrow 0,
 \end{equation}
 where $\mathcal{T}$ is supported only at the nodal point(s).
 Also, it is not hard show that each $\mathcal{T}_{p_i}$ is a one dimensional vector space over $\mathbb{C}$ and that $h^0(\mathcal{T}) = (n-1)$.
 From the exact sequence \ref{canonical exact sequence}, we have 
\begin{eqnarray}
    \chi(\mathcal{O}_C) & = & \sum_{j=1}^n \chi(\mathcal{O}_{C_j}) - \chi(\mathcal{T}) \nonumber \\
    & = & n - \sum_{j=1}^n g_j - (n-1)
    \nonumber \\
    & = & 1 - \sum_{j=1}^n g_j
\end{eqnarray}
So if we let $p_a(C) = 1 - \chi(\mathcal{O}_C)$ be the arithmetic genus of $C$, then $p_a(C) = \sum\limits_{j=1}^n g_j$.

Now suppose $E$ is a vector bundle of rank $r$ on $C$. Tensoring the exact sequence (\ref{canonical exact sequence}) with $E$,  denoting $E_{|_{C_j}}$ by $E_j$ and the cokernel by $\mathcal{T}_E$, we obtain 
\begin{equation}\label{canonical sequence for the vector bundle E}
  0 \rightarrow E \rightarrow \bigoplus_{j=1}^n E_j \rightarrow \mathcal{T}_E \rightarrow 0. 
\end{equation}
We also have 
\[
\chi(E) = \sum_{j=1}^n \chi(E_j) - r(n-1).
\]

Further, for each $j \in \{1,\dots,n\}$, we have the following short exact sequences :
\begin{equation} \label{Exact sequence defining I_C_j}
 0 \rightarrow \mathcal{I}_{C_j} \rightarrow \mathcal{O}_{C} \rightarrow \mathcal{O}_{C_j} \rightarrow 0,
\end{equation}
where 
\begin{equation*} 
(\mathcal{I}_{C_j})_q = 
   \begin{cases}
 \mathcal{O}_{C,q} ~ \text{if} ~ q \in C \setminus C_j,\\      
  0 ~~\text{if}~~ q \in C_j \setminus (C_j \cap \lbrace p_1,\dots, p_{n-1} \rbrace),  \\
 t \mathcal{O}_{C,q} ~~ \text{if} ~~q \in C_j \cap \lbrace p_1,\dots, p_{n-1} \rbrace
\end{cases}  
 \end{equation*}
 and $t \in \mathfrak{m}_{C, q} \setminus \mathfrak{m}^2_{C, q}$ such that $C_j \cap \text{Spec}(\mathcal{O}_{C,q})$ is defined by $t = 0$. \\
\begin{notation}
  Let $E$ be a vector bundle on $C$ and $E_j$ be $E_{|_{C_j}}$. From now on, by $E_1(-p_0-p_1)$ and $E_n(-p_{n-1}-p_n)$ we mean $E_1(-p_1)$ and $E_n(-p_{n-1})$ respectively. There are no such points as $p_0$ and $p_n$ on $C$. But this convention helps in stating some of the subsequent results in a compact way. We also make an explicit mention of the sheaves $E_1(-p_1)$ and $E_n(-p_{n-1})$ if required.
 \end{notation} 
 With this notation, for each $j \in \{1,\dots,n\}$, we also have the following inclusions: 
 \begin{equation} \label{Inclusion involving I_C_j}
     \bigoplus_{\substack{i= 1 \\ i \neq j}}^ n \mathcal{O}_{C_i}(-p_{i-1}-p_i) \hookrightarrow \mathcal{I}_{C_j}
 \end{equation}

\begin{definition}
 Let $F$ be a coherent sheaf of $\mathcal{O}_C$-modules. We call $F$ a pure sheaf of dimension one if for every proper $\mathcal{O}_C$-submodule  $G \subset F$ and $G \neq 0$, the dimension of support of $G$ is equal to one.
\end{definition}
Vector bundles on $C$ are examples of pure sheaves of dimension one. Suppose $F$ is a pure sheaf of dimension one on $C$. Let $F_j = \frac{F_{|_{C_j}}}{\text{Torsion}(F_{|_{C_j}})}$ for each $j$, where $\text{Torsion}(F_{|_{C_j}})$ is the torsion subsheaf of $F_{|_{C_j}}$. Then $F_j$, if non-zero, is torsion-free and hence locally free on $C_j$. Let $r_j$ denote the rank of $F_j$. Also let $d_j$ denote the degree of $F_j$ for each $j$. Then we call the n-tuples $(r_1,\dots,r_n)$ and $(d_1,\dots,d_n)$ respectively the \textit{multirank} and \textit{multidegree} of $F$.
\begin{definition}\label{polarization}
 Let $w = (w_1,\dots,w_n)$ be a n-tuple of rational numbers such that $0 < w_j < 1$ for each $j$ and $\sum\limits_{j=1}^n w_j =1$. We call such a n-tuple a polarization on $C$.
\end{definition}

\begin{definition}\label{polarized slope}
 Suppose $F$ is a pure sheaf of dimension one on $C$ of multirank $(r_1,\dots,r_n)$. Then the slope of $F$ with respect to a polarization $w$, denoted by $\mu_w(F)$, is defined by $\mu_w(F) = \frac{\chi(F)}{\sum_{j=1}^n w_jr_j}$.
\end{definition}

\begin{definition}
 Let $E$ be a vector bundle defined on $C$. Then $E$ is said to be $w$-semistable (resp. $w$-stable) if for any proper subsheaf $F \subset E$ one has $\mu_w(F) \leq \mu_w(E)$ (resp. $\mu_w(F) < \mu_w(E)$). If $E$ is not $w$-semistable for any polarization $w$ on $C$, we call it a strongly unstable bundle.
\end{definition}
The following result from \cite[Theorem-1, Steps-1,2]{Bigas91} will be useful in Theorems \ref{general result on stability} and \ref{a theorem on strongly unstable kernel bundles}.
\begin{theorem}\label{theorem of Bigas}
 Suppose $C$ is a chain-like curve and $E$ is a vector bundle on $C$ of rank $r$. Let $E_i$ denote the restriction of $E$ to the component $C_i$ for $i = 1,\dots,n$ and $w$ denote a polarization on $C$. Then 
 \begin{enumerate}
     \item[(i)] if $E$ is $w$-semistable, $\chi$ is the Euler characteristic of $E$ and $\chi_i$ the Euler characteristic of $E_i$ for each $i$, then
     \begin{equation}\label{polarization inequalities}
         (\sum_{j=1}^i w_j)\chi - \sum_{j=1}^{i-1}\chi_j + r(i-1) \leq \chi_i \leq (\sum_{j=1}^i w_j)\chi - \sum_{j=1}^{i-1}\chi_j + ri,
     \end{equation}
     where $i$ varies from $1$ to $n-1$;
     \item[(ii)] if $E_i$ is a semistable vector bundle for $i = 1,\dots,n$ and the Euler characteristics of $E$ and $E_i$ satisfy the inequalities (\ref{polarization inequalities}), then $E$ is $w$-semistable. Further, if one of the $E_i$ is stable, then $E$ is $w$-stable.
 \end{enumerate}
\end{theorem}
\section{The w-semistable kernel bundles}
Let $C$ be a chain-like curve defined in the previous section. Let $(E, V)$ be a generated pair on $C$ and $E_j$ denote the restriction of $E$ to $C_j$. Then we have a natural morphism 
\begin{align*}
\rho_j : H^0(E) \rightarrow H^0(E_j)
\end{align*} 
defined by $\rho_j (s) = s \vert_{E_j}$ for all $s \in H^0(E)$. We denote $\rho_j(V)$ by $V_j$. The pair $(E_j, V_j)$ is said to be {\em of type} $(r, d_j, k_j )$, if ${\rm rk}(E_j) = r$, $\deg (E_j) = d_j$ and $\dim(V_j) = k_j$. Let $d= \sum\limits_{j=1}^n d_j$ and $\dim (V) = k$. We say that the pair $(E, V)$ is {\em of type} $(r, d, k)$.
\begin{lemma}\label{lemma about generated pair and ontoness}
Let $(E, V)$ be a generated pair on a chain-like curve $C$. Then
\begin{enumerate}
\item[(a)] $(E_j, V_j)$ is a generated pair on $C_j$, for $j = 1,\dots,n$.
\item[(b)] $h^0(E) = \sum\limits_{j=1}^n h^0(E_j) -(n-1)r$,
where $r$ denotes the rank of $E$.
\end{enumerate}
\end{lemma}
\begin{proof}
({\em a}) 
Since $(E, V)$ is a generated pair on $C$, we have  
\[
0 \rightarrow M_{E, V} \rightarrow V \otimes \mathcal{O}_C \rightarrow  E \rightarrow 0.
\]
Tensoring with $\mathcal{O}_{C_j}$, we get
\[
 M_{E, V} \otimes \mathcal{O}_{C_j}\rightarrow (V \otimes \mathcal{O}_C) \otimes \mathcal{O}_{C_j} \rightarrow E \otimes \mathcal{O}_{C_j} \rightarrow 0.
\]
 Since ${\rm Tor}^1(\mathcal{O}_{C_j}, E) \cong {\rm Tor}^1(E, \mathcal{O}_{C_j}) = 0$, the above exact sequence is also left exact. 
We then have the following commutative diagram :
\[
\begin{tikzcd}
  0 \arrow[r] & M_{E, V} \otimes \mathcal{O}_{C_j}  \arrow[r, ] & V \otimes \mathcal{O}_{C_j} \arrow[d, ] \arrow[r] & E \otimes \mathcal{O}_{C_j} \arrow[equal]{d} \arrow[r] & 0 \\
   &  & V_j \otimes \mathcal{O}_{C_j} \arrow[r] & E_j \ar[r] & 0
\end{tikzcd}
\]
 Therefore, the pair $(E_j, V_j)$ is a generated pair on $C_j$.\\
  (b) Consider the exact sequence (\ref{canonical sequence for the vector bundle E}).
Passing to cohomology, we have
\begin{equation}
0 \rightarrow H^0( E) \rightarrow \bigoplus\limits_{j=1}^n H^0(E_j) \rightarrow H^0(\mathcal{T}_E) \rightarrow H^1(E) \rightarrow \dots .
\end{equation}
Since $E_j$ is globally generated for each $j$, we have
 \begin{equation}
0 \rightarrow H^0( E) \rightarrow  \bigoplus\limits_{j=1}^n H^0(E_j) \rightarrow H^0(\mathcal{T}_E) \rightarrow 0.
\end{equation}
\noindent
Therefore, $h^0(E) =  \sum\limits_{j=1}^n h^0(E_j) -(n-1)r$.
\\ \\
\end{proof}
Consider the exact sequence 
\begin{equation}\label{exact sequence for O_C1}
0 \rightarrow \mathcal{I}_{C_j}   \rightarrow \mathcal{O}_C \rightarrow \mathcal{O}_{C_j} \rightarrow 0.
 \end{equation}
 Tensoring it with $E$, we have 
\begin{equation}\label{equation to H(E2) is onto}
0 \rightarrow  E \otimes \mathcal{I}_{C_j} \rightarrow E \rightarrow E_j \rightarrow 0.
\end{equation}
We denote $E \otimes \mathcal{I}_{C_j}$ by $F_j$. Passing to the long exact sequence in the cohomology, we obtain
\begin{equation}
0 \rightarrow H^0(F_j) \rightarrow H^0(E) \xrightarrow{\rho_j} H^0(E_j) \rightarrow \dots
\end{equation} 
Consider the restriction map $\rho_{j_{|_V}} : V \rightarrow V_j$. Then $\text{Ker}(\rho_{j_{|_V}})$ will be equal to $V \cap H^0(F_j)$.
\noindent
From the inclusions (\ref{Inclusion involving I_C_j}), the following can be concluded:
\begin{equation*}
    \bigoplus_{\substack{i=1 \\ i \neq j}}^n \bigg(V \cap H^0(E_i(-p_{i-1} - p_i))\bigg) \subseteq \text{Ker}(\rho_{j_{|_V}}).
\end{equation*}
\begin{lemma}\label{lemma restriction of kernel bundle is not semi-stable}
Let $(E, V)$ be a generated pair on a chain-like curve $C$.  Suppose that
\begin{enumerate}
    \item[(i)] there exists $l \in \{1,\dots,n-1\}$ and $m \in \{2,\dots,n\}$ such that $l < m$, $V \cap H^0(E_l(-p_{l-1} - p_l)) \neq 0$ and $V \cap H^0(E_m(-p_{m-1} - p_m)) \neq 0$ and
    \item[(ii)] $E_l$ and $E_m$ are semistable vector bundles on $C_l$ and $C_m$ respectively.
\end{enumerate}
 Then 
\begin{enumerate}
\item[(a)] $\text{Ker}(\rho_{j_{|_V}}) \neq 0$ for each $j \in \{1,\dots,n\}$,
\item[(b)] $M_{E_j, V_j}$ is a non-trivial quotient of the restriction $M_{E,V} \otimes \mathcal{O}_{C_j}$ of the kernel bundle $M_{E, V}$ to $C_j$ for each $j \in \{1,\dots,n\}$,
\item[(c)] the degrees $d_l$ and $d_m$ of $E_l$ and $E_m$ respectively are greater than or equal to $r$,
\item[(d)] the restriction $M_{E,V} \otimes \mathcal{O}_{C_j}$ of the kernel bundle $M_{E,V}$ to $C_j$ is not a semistable vector bundle, when $j \in \{l,m\}$.
\end{enumerate}
\begin{proof}
({\em a}) This follows from the arguments written just above the statement of Lemma \ref{lemma restriction of kernel bundle is not semi-stable}.\\ \\
({\em b}) Consider the following commutative diagram :
\[
\begin{tikzcd}
  & 0 \arrow[d, ]  & 0 \arrow[d, ]  & 0 \arrow[d, ]  \\
   0 \arrow[r] & {Ker(\rho_{j_{|_V}})} \otimes \mathcal{O}_{C_j} \arrow[equal]{d} \arrow[r, ] & M_{E,V}\otimes \mathcal{O}_{C_j} \arrow[d, ] \arrow[r] & M_{E_j, V_j} \arrow[d, ] \arrow[r] & 0 \\
    0 \arrow[r] & {Ker(\rho_{j_{|_V}})} \otimes \mathcal{O}_{C_j} \arrow[d, ] \arrow[r, ] & V\otimes \mathcal{O}_{C_j} \arrow[d, ] \arrow[r,] & V_j \otimes \mathcal{O}_{C_j} \arrow[d, ] \arrow[r] & 0 \\
     & 0  \arrow[r, ] & E_j \arrow[d, ] \arrow[equal]{r} & E_j \arrow[d, ] \arrow[r] & 0\\
  & &0  & 0   
\end{tikzcd}
\]
Since $\text{Ker}(\rho_{j_{|_V}}) \neq 0$ for $j \in \{1,\dots,n\}$, this diagram shows that $M_{E_j,V_j}$ is a non-trivial quotient of $M_{E,V} \otimes \mathcal{O}_{C_j}$ for each $j$. \\ 
 
\noindent ({\em c}) We prove this for $j = l$. The other case follows similarly. 

\noindent The existence of a nonzero section $s_l \in V \cap H^0(E_l(-p_{l-1} - p_{l}))$ implies that $Z_0(s_l)$, which is the divisor of zeros of $s_l$, is an effective divisor on $C_l$. Let $\mathcal{O}_{C_l}(Z_0(s_l))$ be the line bundle corresponding to $Z_0(s_l)$. We then have the exact sequence
\[
0 \rightarrow \mathcal{O}_{C_l} \rightarrow \mathcal{O}_{C_l}(Z_0(s_l)) \rightarrow T \rightarrow 0,
\]
where $T$ is a torsion sheaf. Therefore, $\deg(Z_0(s_l)) \geqslant 1$. Since $E_l$ is semi-stable and $\mathcal{O}_{C_l}(Z_0(s_l)) \subset E_l$, we have
\[
\frac{d_l}{r} = \mu(E_l) \geqslant \mu(\mathcal{O}_{C_l}(Z_0(s_l))) = \deg(Z_0(s_l)) \geqslant 1.
\] 
This implies that $d_l \geq r$.
\\ 

\noindent ({\em d}) Here again we prove the result for $j = l$. Computing the slope of $M_{E,V} \otimes \mathcal{O}_{C_l}$, we get
\[
\mu(M_{E,V} \otimes \mathcal{O}_{C_l}) = \frac{\deg(M_{E,V} \otimes \mathcal{O
}_{C_l})}{{\rm rk}(M_{E,V} \otimes \mathcal{O}_{C_l})} = \frac{- \deg(E_l)}{k-r} = \frac{-d_l}{k-r} < 0,
\] 
since $d_l \geqslant r > 0$. \\
On the other hand, the subbundle $\text{Ker}(\rho_{l_{|_V}}) \otimes \mathcal{O}_{C_l}$ of $M_{E,V} \otimes \mathcal{O}_{C_l}$ has zero slope as it is a trivial bundle. Thus,  $\text{Ker}(\rho_{l_{|_V}}) \otimes \mathcal{O}_{C_l}$ is a destabilizing subbundle of $M_{E,V} \otimes \mathcal{O}_{C_l}$. Therefore $M_{E,V} \otimes \mathcal{O}_{C_l}$ is not semi-stable.  

\end{proof}
 \end{lemma}

\begin{corollary}\label{restrictions of the kernel bundle are also kernel bundles}
Let $C$ be a chain-like curve. Suppose $(E,V)$ is a generated pair such that $\text{Ker}(\rho_{j_{|_V}}) = 0$ for $j = 1,\dots,n$. Then $M_{E,V} \otimes \mathcal{O}_{C_j} \cong M_{E_j,V_j}$ for $j=1,\dots,n$ and conversely. 

\begin{proof}
 The morphism $\rho_{j_{|_V}} : V \rightarrow V_j$ is an isomorphism if and only if Ker($\rho_{j_{|_V}}) = 0$. So, from the commutative diagram in Lemma \ref{lemma restriction of kernel bundle is not semi-stable}, the previous statement is true if and only if $M_{E,V} \otimes \mathcal{O}_{C_j} \cong M_{E_j,V_j}$ for $j=1,\dots,n$.
\end{proof}
\end{corollary}

\begin{remark}\label{remark that conclude lemma restriction of kernel bundle is not semi-stable}
\begin{enumerate}
    \item As $E_j$ is a globally generated vector bundle for each $j$, its degree $d_j$ will be at least $0$. Lemma \text{\ref{lemma restriction of kernel bundle is not semi-stable}} says that $d_l$ and $d_m$ are strictly bigger than $0$, if $l$ and $m$ are as in the statement of the lemma.
    \item Suppose $V \cap H^0(E_j(-p_{j-1} -p_j)) \neq 0$ for each $j \in \{1,\dots,n\}$. Then $d_j \geq r$ and $M_{E,V} \otimes \mathcal{O}_{C_j}$ is not semistable for each $j \in \{1,\dots,n\}$.
\end{enumerate}
\end{remark} 

\begin{lemma}\label{lemma on existence of polarization}
 Let $C$ be a chain-like curve. Suppose $\mathcal{V}$ is a nonzero finite dimensional vector space over $\mathbb{C}$ and $\mathcal{M}$ is a subbundle of the trivial bundle $\mathcal{V} \otimes \mathcal{O}_C$ on $C$ such that $\text{rk}(\mathcal{M}) = m$, $\chi(\mathcal{M}) < 0$ and the quotient $\mathcal{E} : = \frac{\mathcal{V} \otimes \mathcal{O}_C}{\mathcal{M}}$ is locally free. Then there exists a polarization $w = (w_1,\dots,w_n)$ such that 
 \begin{equation*}
  (\sum_{j=1}^iw_j)\chi - \sum_{j=1}^{i-1}\chi_j + m(i-1) \leq
  \chi_i \leq (\sum_{j=1}^iw_j)\chi - \sum_{j=1}^{i-1}\chi_j + mi,
 \end{equation*} 
 where $\chi := \chi(\mathcal{M})$, $\chi_j := \chi(\mathcal{M}_{\vert_{C_j}})$ and $i$ varies from $1$ to $n-1$.
 \begin{proof}
  The hypothesis that the cokernel $\mathcal{E}$ is locally free implies (by arguments similar to Lemma \ref{lemma about generated pair and ontoness}(a)) that $\mathcal{M}_{\vert_{C_j}}$ is a subbundle of $\mathcal{V} \otimes \mathcal{O}_{C_j}$ for each $j$. So $\chi_j < 0$ for each $j$. Also since $\chi = \sum\limits_{j=1}^n \chi_j - (n-1)m$, we can conclude that for each $j$, $0< \frac{\chi_j}{\chi} < 1$, and for $i \in \{1,\dots,n-1\}$, $0 < \frac{ \sum\limits _{j=1}^i\chi_j - m(i-1)}{\chi} < 1$ and $0 < \frac{\sum\limits_{j=1}^i\chi_j-mi}{\chi} < 1$.\\
  
 \noindent Now, we will first prove the existence of $w_1 \in (0,1) \cap \mathbb{Q}$ such that 
  \begin{equation*}
     w_1\chi \leq \chi_1 \leq w_1\chi + m,
  \end{equation*}
  and then by induction, prove the existence of the remaining $w_i's$. So first we want $w_1 \in (0,1) \cap \mathbb{Q}$ such that \begin{equation}\label{inequalities for w1}
        \begin{cases}
         \chi_1 \leq w_1 \chi + m \\
         \chi_1 \geq w_1 \chi.
        \end{cases}
  \end{equation}
  The inequalities (\ref{inequalities for w1}) are equivalent to
  \begin{equation}\label{another form of inequalities for w1}
    \chi_1 - m \leq w_1\chi \leq \chi_1.  
  \end{equation}
  Multiplying the inequalities (\ref{another form of inequalities for w1}) by ($\frac{1}{\chi}$) and using the hypothesis that $\chi < 0$, we get 
  \begin{equation}\label{third inequality for w1}
    \frac{\chi_1 - m}{\chi} \geq w_1 \geq \frac{\chi_1}{\chi}.
  \end{equation}
  We know that the numbers $\frac{\chi_1 - m}{\chi}$ and $\frac{\chi_1}{\chi}$ are between $0$ and $1$. So it is always possible to choose $w_1 \in (0,1) \cap \mathbb{Q}$ satisfying inequalities (\ref{inequalities for w1}).\\
  
  \noindent   \noindent Now assume $i \geq 1$ to be a positive integer strictly less than $n-1$. Suppose that $w_i$'s are chosen from $(0,1) \cap \mathbb{Q}$ such that $\sum\limits_{j=1}^iw_j \in (0,1) \cap \mathbb{Q}$,
   \begin{equation*}
     \frac{\sum\limits _{j=1}^i\chi_j - mi}{\chi} \geq \sum\limits_{j=1}^{i}w_j \geq \frac{\sum\limits_{j=1}^i\chi_j - m(i-1)}{\chi} 
  \end{equation*}
  or equivalently
  \begin{equation*}
  (\sum\limits_{j=1}^iw_j)\chi - \sum\limits_{j=1}^{i-1}\chi_j + m(i-1) \leq
  \chi_i \leq (\sum_{j=1}^iw_j)\chi - \sum\limits_{j=1}^{i-1}\chi_j + mi.
 \end{equation*}
  We want to prove that there exists $w_{i+1} \in (0,1) \cap \mathbb{Q}$ such that $\sum\limits_{j=1}^{i+1}w_j \in (0,1) \cap \mathbb{Q}$ and
  \begin{equation}\label{inequalities for w_(i+1)}
  (\sum\limits_{j=1}^{i+1}w_j)\chi - \sum\limits_{j=1}^{i}\chi_j + mi \leq
  \chi_{i+1} \leq (\sum\limits_{j=1}^{i+1}w_j)\chi - \sum\limits_{j=1}^{i}\chi_j + m(i+1).
 \end{equation}
  Again proceeding in the same way as we did for $w_1$, the inequality (\ref{inequalities for w_(i+1)}) is equivalent to 
 \begin{equation}\label{second inequality for w_(i+1)}
   \frac{\sum\limits_{j=1}^{i+1}\chi_j - m(i+1)}{\chi} \geq \sum\limits_{j=1}^{i+1}w_j \geq  \frac{\sum\limits_{j=1}^{i+1}\chi_j - mi}{\chi}.   
 \end{equation}
 Since $\frac{\sum\limits_{j=1}^{i+1}\chi_j - m(i+1)}{\chi}$ and $\sum\limits_{j=1}^{i+1}w_j \geq  \frac{\sum\limits_{j=1}^{i+1}\chi_j - mi}{\chi}$ are numbers between $0$ and $1$, we can conclude that $\sum\limits_{j=1}^{i+1}w_j$ is in $(0,1) \cap \mathbb{Q}$. Choose any $w_{i+1}$ satisfying (\ref{second inequality for w_(i+1)}); it is enough to prove that $w_{i+1} > 0$. The inequalities (\ref{second inequality for w_(i+1)}) are equivalent to
  \begin{equation}\label{fourth inequlity for w_(i+1)}
     \frac{\sum\limits_{j=1}^{i+1}\chi_j - m(i+1)}{\chi} - \sum\limits_{j=1}^iw_j \geq w_{i+1} \geq  \frac{\sum\limits_{j=1}^{i+1}\chi_j - mi}{\chi}-\sum\limits_{j=1}^iw_j. 
     \end{equation}
    Since $\sum\limits_{j=1}^iw_j \leq \frac{\sum\limits_{j=1}^i\chi_j - mi}{\chi}$, we have
    \begin{eqnarray}
      w_{i+1} & \geq & \frac{\sum\limits_{j=1}^{i+1}\chi_j - mi}{\chi} - \frac{\sum\limits_{j=1}^i\chi_j - mi}{\chi} \nonumber \\
      & = & \frac{\chi_{i+1}}{\chi} \nonumber \\
      & > & 0. \nonumber
    \end{eqnarray}
    Once we have chosen $w_1,\dots,w_{n-1}$ in this way, we can define $w_n$ to be $1-\sum\limits_{j=1}^{n-1}w_j$. This gives us the required polarization and completes the proof.
 \end{proof}
\end{lemma}
\begin{remark}
 The above lemma applies, in particular, to $\mathcal{V} \otimes \mathcal{O}_C$ itself. More precisely, any trivial bundle is always a proper subbundle of a higher rank trivial bundle. Also Euler characteristic of any trivial bundle over $C$ is some positive integer multiple of $(1-p_a(C))$, and so, is negative ($p_a(C)$ is at least $4$ because $C$ has at least $2$ components and $g_j \geq 2$ for each $j$). Therefore, from the above lemma, we can conclude that given any trivial bundle $\mathcal{V} \otimes \mathcal{O}_C$, there always exists a polarization $w$ such that $\mathcal{V} \otimes \mathcal{O}_C$ is $w$-semistable. This will imply that if $\mathcal{M}$ is any proper subsheaf of $\mathcal{V} \otimes \mathcal{O}_C$ then $\chi(M) < 0$. So the assumption $\chi(\mathcal{M}) < 0$ in the statement of the lemma is redundant. But we have kept it in the lemma as we have not proved the $w$-semistability of $\mathcal{V} \otimes \mathcal{O}_C$ separately.
\end{remark}
Now, suppose $(E,V)$ is a generated pair on $C$. Then $M_{E,V}$ is a subbundle of the trivial bundle $V \otimes \mathcal{O}_C$. By the defining exact sequence (\ref{defining sequence for the kernel bundle}) of $M_{E,V}$ we have
\[
\begin{split}
\chi(M_{E, V}) & = \chi (V \otimes \mathcal{O}_C) - \chi(E) \\
                        & = k (1 - p_a(C)) - d - r(1 - p_a(C)) \\
                        & = (k- r)(1 - p_a(C)) -d,
\end{split}
\]
where $d = \sum\limits_{i=1}^n d_i$. This, in particular, means that $\chi(M_{E,V}) < 0$. So, as a particular case of Lemma \ref{lemma on existence of polarization}, we have the following lemma:
\begin{lemma}\label{semistability related lemma}
 Let $C$ be a chain-like curve. Suppose $(E,V)$ is a generated pair on $C$.
 Then there exists a polarization $w = (w_1,\dots,w_n)$ such that 
 \begin{equation*}
  (\sum_{j=1}^iw_j)\chi - \sum_{j=1}^{i-1}\chi_j + (k-r)(i-1) \leq
  \chi_i \leq (\sum_{j=1}^iw_j)\chi - \sum_{j=1}^{i-1}\chi_j + (k-r)i,
 \end{equation*} 
 where $\chi := \chi(M_{E,V})$, $\chi_j := \chi({M_{E,V}}_{\vert_{C_j}})$, $i$ varies from $1$ to $n-1$ and $(k-r)$ is the rank of $M_{E,V} =$
 rank of ${M_{E,V}}_{\vert_{C_j}}$.

\end{lemma}
\begin{theorem}\label{general result on stability}
 Let $C$ be a chain-like curve and $\mathcal{M}$ be as in Lemma  \ref{lemma on existence of polarization}. Suppose $\mathcal{M}_{|_{C_j}}$ is semistable for each $j$. Then there exists a polarization $w$ such that $\mathcal{M}$ is $w$-semistable. Further, suppose $\mathcal{M}_{|_{C_j}}$ is stable for some $j$. Then $\mathcal{M}$ is $w$-stable.
 \begin{proof}
  By Lemma \ref{lemma on existence of polarization}, there exists a polarization $w = (w_1,\dots,w_n)$ such that 
 \begin{equation*}
  (\sum_{j=1}^iw_j)\chi - \sum_{j=1}^{i-1}\chi_j + m(i-1) \leq
  \chi_i \leq (\sum_{j=1}^iw_j)\chi - \sum_{j=1}^{i-1}\chi_j + mi,
 \end{equation*} 
 where $\chi := \chi(\mathcal{M})$, $\chi_j := \chi(\mathcal{M}_{\vert_{C_j}})$ and $i$ varies from $1$ to $n-1$. 
 Since ${\mathcal{M}}_{|_{C_j}}$ is semistable for each $j$ and the polarization $w$ satisfies the above set of inequalities, by Theorem \ref{theorem of Bigas}, the result follows.
 \end{proof}
\end{theorem}

\begin{corollary}\label{stability of the kernel bundle}
 Let $C$ be a chain-like curve and $(E,V)$ be a generated pair on $C$.
 If for each $j$, ${M_{E,V}}_{\vert_{C_j}}$ is semistable, then there exists a polarization $w$ such that $M_{E,V}$ is $w$-semistable. Further, if for some $j$, ${M_{E,V}}_{\vert_{C_j}}$ is stable, then $M_{E,V}$ is $w$-stable.
 \begin{proof}
  Follows directly from Lemma \ref{semistability related lemma} and Theorem \ref{general result on stability}.
 \end{proof}
\end{corollary}
\begin{remark}
Suppose $\text{Ker}(\rho_{j_{|_V}}) = 0$ for each $j$. 
 If $M_{E_j,V_j}$ is semistable for each $j$, then there exists a polarization $w$ such that $M_{E,V}$ is $w$-semistable. This follows from corollary \ref{restrictions of the kernel bundle are also kernel bundles} which coincides with \cite[Theorem 3.2]{Brivio-Favale2020} for $n=2$.
\end{remark}


\begin{example} \label{Example_where_ker(rho_j_V)_is_nonzero}
  Suppose $C$ is a general nodal curve with two smooth components $C_1$ of genus $g_1 \geq 2$ and $C_2$ of genus $g_2 \geq 2$ intersecting at a node $p$. Since $C$ is general, both of its components are general in their corresponding moduli spaces. Let $r$ be a fixed positive integer, $m=1,~ d_1 = 2rg_1~\text{and}~d_2 = 2rg_2$. Then since $g_j \geq 2~\text{and}~d_j = 2rg_j$ for each $j$, and $m=1$, we can conclude that $d_j \geq rg_j + m$ for each $j$. So, by the proof of \cite[Theorem 3.3]{Brambila-Newstead2019}, for each $j$, there exists a generated pair $(E_j,V_j)$ on $C_j$ such that $E_j$ is a stable vector bundle on $C_j$, $\text{dim}(V_j) = r+1$ and $h^0(E_j) = d_j+r(1-g_j) = rg_j+r$.\\
  
  \noindent Now consider the canonical maps $V_j \xrightarrow{\theta_j} \frac{E_{j,p}}{\mathfrak{m}_{j,p}E_{j,p}}$, where $\mathfrak{m}_{j,p}$ is the maximal ideal in the local ring $\mathcal{O}_{C_{j,p}}$. The fact that $(E_j,V_j)$ is a generated pair for each $j$ means that $\theta_j$ is surjective. Since the vector space $V_j$ is $r+1$ dimensional and the space $\frac{E_{j,p}}{\mathfrak{m}_{j,p}E_{j,p}}$ is $r$ dimensional, the $\text{Ker}(\theta_j)$ will be one dimensional. Let $\{s_{r+1}\}~\text{and}~\{t_{r+1}\}$ be the bases for $\text{Ker}(\theta_1)$ and $\text{Ker}(\theta_2)$ respectively. Extend these to the bases $\{s_1,\dots,s_{r+1}\}$ and $\{t_1,\dots,t_{r+1}\}$ of $V_1$ and $V_2$ respectively. Let $\sigma:\frac{E_{1,p}}{\mathfrak{m}_{1,p}E_{1,p}} \rightarrow \frac{E_{2,p}}{\mathfrak{m}_{2,p}E_{2,p}}$ be the isomorphism defined by  $\bar{s_j} \mapsto \bar{t_j}$ for $j=1,\dots,r$, where $\bar{s_j}$ and $\bar{t_j}$ are images of $s_j$ and $t_j$ in $\frac{E_{1,p}}{\mathfrak{m}_{1,p}E_{1,p}}$ and $\frac{E_{2,p}}{\mathfrak{m}_{2,p}E_{2,p}}$ respectively. Now, suppose $E$ is the vector bundle on $C$ that corresponds to the triple $(E_1,E_2,[\sigma])$ (see \cite[Section 2]{nagaraj-seshadri96} for details on the correspondence between torsion free sheaves and triples). Then it is clear that $\lbrace (s_1,t_1), (s_2, t_2), \dots, (s_r, t_r), (s_{r+1},0), (0, t_{r+1})\rbrace
\subset H^0(E) $. Let the set $\lbrace (s_1,t_1), \dots, (s_r, t_r), (s_{r+1},0), (0, t_{r+1})\rbrace$ be denoted by $\mathcal{S}$ and $V = \text{Span}(\mathcal{S})$. Then $\text{dim}(V) = r+2$ and $(E,V)$ is a generated pair on $C$. Since $g_j \geq 2$ for each $j$, we can conclude that $r+2 \leq \text{min}((rg_1+r),(rg_2+r))$. So we have $r < (r+2) \leq \text{min}(h^0(E_1),h^0(E_2))$ and $(r+2)$ sections generate $E$. Therefore, by \cite[Remark 3.2.1]{Brivio-Favale2020}, a general pair $(E,W)$ on $C$ such that $\text{dim}(W) = r+2$ is generated and satisfies $W \cap H^0(E_j(-p)) = \{0\}$ for each $j$. In particular, this means that $\text{Ker}(\rho_j{_{|_W}}) = 0$ for each $j$. 
\end{example}

\begin{theorem}
Let $C$ be a general chain-like curve, and $(L, V)$ a generated pair on $C$ such that $(L_j, V_j)$ is a general linear series for each $j$. Suppose for each $j$, $\text{Ker}(\rho_{j_{|_V}}) = 0$. Then there exists a polarisation $w$ such that $M_{E, V}$ is $w$-semistable. 
\begin{proof}
Since $\text{Ker}(\rho_{j_{|_V}}) = 0$, by Corollary \ref{restrictions of the kernel bundle are also kernel bundles} we have $M_{L_j, V_j} \cong {M_{L, V}}_{\vert_{C_j}}$. From the hypothesis on the generated pair $(L_j, V_j)$, the corresponding kernel bundle $M_{L_j, V_j}$ on $C_j$ is semistable for each $j$ (\cite[Theorem 5.1]{Usha15}). Therefore, ${M_{L, V}}_{\vert_{C_j}}$ is semistable for each $j$. Therefore, by Corollary \ref{stability of the kernel bundle}, there exists a polarisation $w$ such that $M_{L, V}$ is $w$-semistable.
\end{proof}
\end{theorem}
\begin{theorem}
Let $C$ be a chain-like curve such that each smooth component $C_j$ of $C$ is a Petri curve. 
Let $(L, V)$ be a generated pair on $C$ with $k \geq 6$, $g_j \geqslant 2 k - 6$.  Suppose $\text{Ker}(\rho_{j_{|_V}}) = 0$ and the linear series $(L_j, V_j)$ is general for $j = 1,\dots,n$. Then $M_{L,V}$ is $w$-stable for a polarisation $w$.
\begin{proof}
First of all note that for each $j$, $M_{L_j, V_j} \cong {M_{L, V}}_{\vert_{C_j}}$ and $k = k_j$, since $\text{Ker}(\rho_{j_{|_V}}) = 0$ for  $1 \leqslant  j \leqslant n$ (see Corollary \ref{restrictions of the kernel bundle are also kernel bundles}). Thus, for each $j$, $k_j \geqslant 6$ and $g_j \geqslant 2 k - 6$. Therefore, for a general linear series $(L_j, V_j)$ satisfying these conditions, the kernel bundle $M_{L_j, V_j}$ associated to  $(L_j, V_j)$ is stable (by \cite[Theorem 6.1]{Usha15}). Therefore,  ${M_{L, V}}_{\vert_{C_j}}$ is stable, for each $j$. Corollary \ref{stability of the kernel bundle} then guarantees the existence of a polarisation $w$ for which $M_{L, V}$ is $w$-stable.
\end{proof}
\end{theorem}

\section{strongly unstable kernel bundles}
In the previous section, we have seen that if the restriction of $M_{E,V}$ to each component is semistable, there exists a polarization $w$ such that $M_{E,V}$ is $w$-semistable. In this section, we provide some sufficient conditions under which $M_{E,V}$ becomes strongly unstable. By Corollary \ref{stability of the kernel bundle}, for $M_{E,V}$ to be strongly unstable, there should exist at least one $j$ such that $M_{E,V} \otimes \mathcal{O}_{C_j}$ is not semistable. Lemma \ref{lemma restriction of kernel bundle is not semi-stable}(d) says that this can be achieved if $V \cap H^0(E_j(-p_{j-1} -p_j)) \neq 0$. So to obtain sufficient conditions for $M_{E,V}$ to be strongly unstable, it is reasonable to assume the hypothesis that $V \cap H^0(E_j(-p_{j-1} -p_j)) \neq 0$, or, more generally that $\text{Ker}(\rho_{j_{|_V}}) \neq 0$ (for at least one $j$). 
\begin{theorem}\label{a theorem on strongly unstable kernel bundles}
 Let $(E,V)$ be a generated pair on a chain-like curve $C$. Then $M_{E,V}$ is strongly unstable, if any of the following conditions hold:
 \begin{enumerate}
     \item[(a)] $V \cap H^0(E_1(-p_{1})) \neq 0$ (or $V \cap H^0(E_n(-p_{n-1})) \neq 0$), $E_1$ (resp. $E_n$) is semistable and $(k-r) < d_1$ (resp. $(k-r) < d_n$).
     \item[(b)] there exists a $j \in \{2,\dots,n-1\}$ such that $V \cap H^0(E_j(-p_{j-1}-p_j)) \neq 0$, $E_j$ is semistable and $(k-r) < \frac{d_j}{2}$.
 \end{enumerate}
 \begin{proof}
 Suppose $M_{E,V}$ is $w$-semistable with respect to a polarization $w$.\\
  \noindent (a) Given $(k-r) < d_1$. By Theorem \ref{theorem of Bigas} (i), the polarization $w$ satisfies the inequalities (\ref{polarization inequalities}). In particular, by using the facts that $\chi_1 = -(k-r)((g_1-1)+d_1)$, $\chi = -(k-r)((p_a(C)-1)+d)$ and the inequality (\ref{third inequality for w1}), we have
  \begin{equation}\label{right side polarization inequality for w1}
    w_1 \geq \frac{(k-r)(g_1-1)+d_1}{(k-r)(p_a(C)-1)+d}.  
  \end{equation}
  On the other hand, since $\text{Ker}(\rho_{1_{|_V}}) \neq 0$, from the commutative diagram in Lemma \ref{lemma restriction of kernel bundle is not semi-stable}(b), we can conclude that $\text{Ker}(\rho_{1_{|_V}}) \otimes \mathcal{O}_{C_1}(-p_1) \subseteq M_{E,V} \otimes \mathcal{O}_{C_1}(-p_1) \subseteq M_{E,V}$ is a non-trivial subsheaf of $M_{E,V}$. So the fact that $M_{E,V}$ is $w$-semistable will imply
  \begin{equation}\label{slope inequality for restriction of the kernel (1)}
 \mu_w ({\rm Ker}(\rho_{1_{|_V}}) \otimes \mathcal{O}_{C_1}(-p_1)) \leq  \mu_w(M_{E,V}).
 \end{equation}
 Denoting $\text{dim}(\text{Ker})(\rho_{1_{|_V}})$ by $t_1$, we have 
 \begin{eqnarray}\label{slope of ker_rho1}
  \mu_w ({\rm Ker}(\rho_{1_{|_V}}) \otimes \mathcal{O}_{C_1}(-p_1)) & = & \frac{t_1\chi(\mathcal{O}_{C_1}(-p_1))}{w_1t_1}  \nonumber \\
  & = & \frac{-g_1}{w_1}.
 \end{eqnarray}
 Also 
 \begin{eqnarray}\label{slope of M{E,V}}
  \mu_w(M_{E,V}) & = & \frac{(k-r)(1-p_a(C))-d}{(k-r)}. 
 \end{eqnarray}
 Using equations (\ref{slope of ker_rho1}) and (\ref{slope of M{E,V}}) in the inequality (\ref{slope inequality for restriction of the kernel (1)}) and simplifying, we get
 \begin{equation}\label{left inequality for w1}
  w_1 \leq \frac{(k-r)g_1}{(k-r)(p_a(C)-1)+d}.     
 \end{equation}
 From the inequalities (\ref{right side polarization inequality for w1}) and (\ref{left inequality for w1}), we get 
 $(k-r)-d_1 \geq 0$ which is a contradiction to the hypothesis that $(k-r) < d_1$. \\
 
 \noindent (b) Given $(k-r) < \frac{d_j}{2}$ for some $j \in \{2,\dots,n-1\}$. Again proceeding as before, from the inequality (\ref{fourth inequlity for w_(i+1)}), we have
 \begin{eqnarray}
  w_{j} & \geq & 
   \frac{(k-r)(\sum\limits_{i=1}^{j}g_i-1)+\sum\limits_{i=1}^{j}d_i}{(k-r)(p_a(C)-1)+d} - \sum\limits_{i=1}^{j-1}w_i \nonumber \\
   & \geq & \frac{(k-r)(\sum\limits_{i=1}^{j}g_i-1)+\sum\limits_{i=1}^{j}d_i}{(k-r)(p_a(C)-1)+d} - \frac{(k-r)(\sum\limits_{i=1}^{j-1}g_i)+\sum\limits_{i=1}^{j-1}d_i}{(k-r)(p_a(C)-1)+d} \nonumber \\
   & = & \frac{(k-r)(g_j-1)+d_j}{(k-r)(p_a(C)-1)+d}. \nonumber
 \end{eqnarray}
 Since $\text{Ker}(\rho_{j_{|_V}}) \otimes \mathcal{O}_{C_j}(-p_{j-1}-p_j) \subseteq M_{E,V} \otimes \mathcal{O}_{C_j}(-p_{j-1}-p_j) \subseteq M_{E,V}$, we have
 \begin{eqnarray}
  w_j & \leq & \frac{(k-r)(g_j+1)}{(k-r)(p_a(C)-1)+d}.
 \end{eqnarray}
 Comparing both the above inequalities involving $w_j$, we get
 $2(k-r)-d_j \geq 0$, which is a contradiction. This completes the proof.
 \end{proof}
\end{theorem}

We now give some basic results which will be used in proving some of the following results of this section. We first mention the following result of Xiao:
\begin{theorem}\label{theorem directly from the paper} 
Let $E$ be a semistable vector bundle of rank $r$ on a smooth irreducible complex projective curve $C$ of genus $g \geq 2$. If $0 \leq \mu(E) \leq 2g -2$, then 
\[
h^0(E) \leq \deg(E)/2 + r.
\]
\begin{proof}
 \cite[Theorem 2.1]{Newstead-Brambila1997}
\end{proof}
\end{theorem}
\begin{remark}
 \begin{enumerate}
     \item Suppose $M_{E,V} \otimes \mathcal{O}_{C_j}$ is semistable for some $j$. Then by the proof of Lemma \ref{lemma restriction of kernel bundle is not semi-stable}(d), $V \cap H^0(E_j(-p_{j-1} -p_j)) =0$. But this will not imply $\text{Ker}(\rho_{j_{|_V}}) =0$ if $n > 2$.
     \item Suppose $M_{E,V} \otimes \mathcal{O}_{C_j}$ is semistable for some $j$ and not a trivial bundle on $C_j$. Also suppose $E_j$ is semistable for such a $j$. Then we claim that $\text{Ker}(\rho_{j_{|_V}}) = 0$. For, if $\text{Ker}(\rho_{j_{|_V}}) \neq 0$, from the commutative diagram in Lemma \ref{lemma restriction of kernel bundle is not semi-stable}(b), the semistability of $M_{E,V} \otimes \mathcal{O}_{C_j}$ will imply that $d_j = 0$. Then by Clifford's theorem, $k_j \leq h^0(E_j) \leq \frac{d_j}{2} + r$, which forces $k_j-r$ to be $0$. So $M_{E_j,V_j} =0$. This means that $M_{E,V} \otimes \mathcal{O}_{C_j} \cong \text{Ker}(\rho_{j_{|_V}}) \otimes \mathcal{O}_{C_j}$, a trivial bundle.
 \end{enumerate}

\end{remark}
 \begin{lemma}\label{lemma for k < d +r}
 Let $(E, V)$ be a generated pair on a chain-like curve $C$ with $E_j$ semistable for all $j$. Suppose there exists $l \in \{1,\dots,n\}$ such that $d_l > 0$. Then $k < d + r$.

\begin{proof}

To prove this, we consider several cases.\\
\begin{enumerate}
\item Suppose that $\mu(E_j) > 2g_j -2$ for each $j$. Then $h^1(E_j) = 0$, since each $E_j$ is semistable. By Lemma \ref{lemma about generated pair and ontoness}(b), we have
\begin{align*}
\begin{split}
k \leqslant h^0(E) & = \sum_{j=1}^nh^0 (E_j) - (n-1)r \\
                              & = d + r(n  - p_a(C)) - (n-1)r \\
                              & < d + r,
\end{split}
\end{align*}
as $p_a(C) \geq 2$.
\item Suppose that $\mu(E_j) \leqslant 2 g_j - 2$ for each $j$. Then
\begin{align*}
\begin{split}
k \leqslant h^0(E) & = \sum_{j=1}^nh^0(E_j) - (n-1)r \\
                              & \leq \frac{d}{2} + nr - (n-1)r~~(\text{by Theorem \ref{theorem directly from the paper}}),  \\
                              & < d + r,
\end{split}
\end{align*}
as $d_j \geq  0$ for each $j$, $d_l > 0$ for some $l$ and $d = \sum\limits_{j=1}^nd_j$.
\item Suppose $A$ is a non-empty proper subset of $\lbrace j \in \mathbb{N} \vert 1 \leqslant j \leqslant n \rbrace$ with $\mu (E_j) \leqslant 2 g_j - 2$. Then by Theorem \ref{theorem directly from the paper} and Riemann-Roch Theorem, we have
\begin{align*}
\begin{split}
k \leqslant h^0(E) & = \sum_{j \in A} h^0(E_j) + \sum_{ t \notin A} h^0(E_t) - (n-1)r \\
                              & \leqslant  \sum_{j \in A} ( d_j/2 + r )  + \sum_{ t \notin A} (d_t + r(1 -g_t))  - (n-1)r  \\
                              & = \sum_{j \in A}d_j/2 + \sum_{ t \notin A} d_t - \sum_{ t \notin A} r g_t + r \\
                              & < \sum_{j \in A}d_j + \sum_{ t \notin A} d_t + r = d + r,
\end{split}
\end{align*}
as $g_t \geqslant 2$. 
\end{enumerate}
In case (3), as $A$ is an arbitrary non-empty proper subset of  $\lbrace j \in \mathbb{N} \vert 1 \leqslant j \leqslant n \rbrace$ such that $\mu (E_j) \leqslant 2 g_j - 2$, it covers all the remaining cases.
Therefore, $k < d + r$.
\end{proof}
\end{lemma}

\begin{remark} \label{This_remark_implies_d_l>0}
 Suppose there exists $l \in \{1,\dots,n\}$ such that $V \cap H^0(E_l(-p_{l-1} - p_l)) \neq 0$ and $E_l$ is semistable. Then by Lemma \ref{lemma restriction of kernel bundle is not semi-stable}(c), $d_l > 0$. 
\end{remark}

Suppose $E$ is a rank $r$ vector bundle and $L$ is a line bundle on $C$. Then we have
\begin{align*}
\begin{split}
\chi(E \otimes L) & = \sum_{j=1}^n \chi(E_j \otimes L) - r (n-1).\\
\end{split}
\end{align*} 
By projection formula, $\chi(E_j \otimes L) = \chi(E_j \otimes L_j) = \chi(E_j)+r\text{deg}(L_j)$, where $L_j := L_{|_{C_j}}$. So 
\begin{equation}\label{chi(E tensor L)}
 \chi(E \otimes L) = \chi(E) + r\text{deg}(L),   
\end{equation}
where by $\deg(L)$ we mean $\sum\limits_{j=1}^n \deg(L_j)$. In particular, we have 
\begin{align}
\chi(M_{E, V}  \otimes L) = (k- r)(1 +  \deg(L) - p_a(C)) -d.
\end{align}

\begin{theorem}\label{theorem for strongly unstable for any L}
Let $(E, V)$ be a generated pair on a chain-like curve $C$ with $ \frac{d}{k-r} > (n-1)$. Suppose that $\text{Ker}({\rho_j}_{|_V}) \neq 0$ for each $j$. Then $M_{E, V} \otimes L$ is strongly unstable, for any line bundle $L$ on $C$. In particular, $M_{E,V}$ is strongly unstable.

\begin{proof}
Suppose that $M_{E, V} \otimes L$ is $w$-semistable with respect to some polarisation $w = (w_1,\dots,w_n)$. Since $\text{Ker}({\rho_j}_{|_V}) \neq 0$ for each $j$, and $M_{E,V} \otimes L$ is $w$-semistable, proceeding as in Theorem \ref{a theorem on strongly unstable kernel bundles}, we obtain 
\begin{eqnarray}\label{left inequality for wj for M(E,V) tensor L}
  (k-r)(g_1) & \geq & w_1[(k-r)(p_a(C)-1-\text{deg}(L))+d] + (k-r)\text{deg}(L_1)     \nonumber \\
  (k-r)(g_n) & \geq & w_n[(k-r)(p_a(C)-1-\text{deg}(L))+d] + (k-r)\text{deg}(L_n) \nonumber \\
  (k-r)(g_j) & \geq & w_j[(k-r)(p_a(C)-1-\text{deg}(L))+d] + (k-r)(\text{deg}(L_j)-1) \nonumber \\ 
  & & \text{for}~j \in \{2,\dots n-1\}. \nonumber
 \end{eqnarray}
 Adding the above inequalities, we get
 \begin{eqnarray}
   (k-r)p_a(C) & \geq & [(k-r)(p_a(C)-1-\text{deg}(L) + \text{deg}(L)- (n-2)] + d, \nonumber 
 \end{eqnarray}
 which implies $0 \geq -(n-1)(k-r) + d$. This contradicts the fact that $\frac{d}{(k-r)} > (n-1)$.

\end{proof}
\end{theorem}
\begin{remark}
 On any nodal curve which is reduced and reducible, if $\mathcal{F}$ is a vector bundle which is semistable with respect to a polarization $w$ then it is not true in general that $\mathcal{F} \otimes L$ is $w$-semistable, for any line bundle $L$ on $C$. It is also not true in general that if $\mathcal{F}$ is strongly unstable, then $\mathcal{F} \otimes L$ is strongly unstable. For example, let $C$ be a two component curve with a single node and let $\mathcal{F} = \mathcal{O}_C$. By Lemma \ref{lemma on existence of polarization} and Theorem \ref{general result on stability}, there exists a polarization $w$ such that $\mathcal{O}_C$ is $w$-semistable. Now suppose $L$ is a line bundle on $C$ which is obtained by gluing line bundles $L_1$ on $C_1$ and $L_2$ on $C_2$ such that $\text{deg}(L_1) =0$ and $\text{deg}(L_2) = (g_1+g_2)$, where $g_1 \geq 2$ is the genus of $C_1$ and $g_2 \geq 2$ is the genus of $C_2$. Then $\chi(L) =1~\text{and}~ \chi_1:=\chi(L_1) <0$. We claim that $L$ is strongly unstable. For, if $L$ is $w$-semistable, then $\chi_i$'s and $w_i$'s have to satisfy the inequalities (\ref{polarization inequalities}). In particular, $w_1$ should be such that $\frac{\chi_1-1}{\chi} \leq w_1 \leq \frac{\chi_1}{\chi}$. But $\frac{\chi_1}{\chi} = \chi_1 < 0$. This implies $w_1 <0$, a contradiction. So even though $\mathcal{O}_C$ is $w$-semistable with respect to a polarization $w$, $\mathcal{O}_C \otimes L = L$ is strongly unstable. Similarly, if we tensor this particular $L$ with its dual $L^{\ast}$, then we have $L$ as strongly unstable but $L \otimes L^{\ast} = \mathcal{O}_C$ as $w$-semistable.\\ 
 
 \noindent Therefore, Theorem \ref{theorem for strongly unstable for any L} becomes significant as it is saying that under suitable conditions, not only $M_{E,V}$, but also $M_{E,V} \otimes L$ is strongly unstable for any line bundle $L$ on a chain-like curve $C$.
\end{remark}
\begin{corollary} \label{strongly unstable for 2-component curve} 
 Suppose $C$ is a reducible nodal curve with two smooth components intersecting at a node $p$. Let $(E,V)$ be a generated pair on $C$. Also let $\text{Ker}({\rho_j}_{|_V}) \neq 0$ and $E_j$ be semistable for each $j$. Then $M_{E, V} \otimes L$ is strongly unstable, for any line bundle $L$ on $C$.
\end{corollary}
\begin{proof}
 Since $C$ has only two components, the exact sequence (\ref{Exact sequence defining I_C_j}) becomes
 \begin{equation*}
     0 \rightarrow \mathcal{O}_{C_i}(-p) \rightarrow \mathcal{O}_C \rightarrow \mathcal{O}_{C_j} \rightarrow 0,
 \end{equation*}
 where $i \neq j$. Tensoring this exact sequence with $E$ and taking the corresponding long exact sequence of cohomology groups, we conclude that $\text{Ker}(\rho_j) = H^0(E_i(-p))$, where $i \neq j$. So $\text{Ker}({\rho_j}_{|_V}) \neq 0$ is equivalent to saying $V \cap H^0(E_i(-p)) \neq 0$. Therefore, by Remark \ref{This_remark_implies_d_l>0} and Lemma \ref{lemma for k < d +r}, we have $d > (k-r)$. The result now follows from Theorem \ref{theorem for strongly unstable for any L}.
\end{proof}
\begin{remark}
 Suppose $L = \mathcal{O}_C$. Then Corollary \ref{strongly unstable for 2-component curve} is nothing but \cite[Theorem 2.4]{Brivio-Favale2020}.
\end{remark}
\begin{corollary}
 Let $C$ be a chain-like curve and $(E, V)$ be a generated pair on $C$. Suppose $ p_a(C) > \frac{(n-2)(k - r)}{r}$, $H^1(E_j) = 0$ and $\text{Ker}({\rho_j}_{|_V}) \neq 0$ for  each $j$. Then $M_{E, V} \otimes L$ is strongly unstable, for any line bundle $L$ on $C$.
 \begin{proof}
 Similar to the proof of the first case of Lemma \ref{lemma for k < d +r}, we have 
  \begin{eqnarray}
    k & \leq & d + r - rp_a(C) \nonumber \\
      & < & d + r - (n-2)(k-r), \nonumber
  \end{eqnarray}
  which implies $(n-1)(k-r)  <  d$.
  The result now follows from Theorem \ref{theorem for strongly unstable for any L}.
 \end{proof}
\end{corollary}
\begin{remark}
In particular, if $E_j$ is semistable, $\mu(E_j) > (2g_j-2)$, $g_j \geq \frac{(k-r)}{r}$ and $\text{Ker}({\rho_j}_{|_V}) \neq 0$ for each $j$, then $M_{E,V} \otimes L$ is strongly unstable for any line bundle $L$ on $C$.
\end{remark}

\begin{example}
 Let $C$ be a reducible nodal curve having two smooth components $C_1$ and $C_2$ intersecting at a node $p$. Let $\mathcal{L}_1$ and $\mathcal{L}_2$ be line bundles on $C_1$ and $C_2$ respectively of sufficiently large degrees such that each $\mathcal{L}_i$ is globally generated and $H^0(\mathcal{L}_i(-p)) \neq 0$ for each $i$. Let $\mathcal{L}$ be the line bundle on $C$ obtained by gluing $\mathcal{L}_1$ and $\mathcal{L}_2$ via a linear isomorphism at the node $p$. Then $\mathcal{L}$ will be globally generated. So $(\mathcal{L}, H^0(\mathcal{L}))$ is a generated pair on $C$. Since $\text{Ker}(\rho_j) = H^0(\mathcal{L}_i(-p))$, where $i \neq j$ (see proof of Corollary \ref{strongly unstable for 2-component curve} for details), we can conclude from our choice of $\mathcal{L}_i$ that 
 $\text{Ker}(\rho_j) \neq 0$ for each $j$. So by Corollary \ref{strongly unstable for 2-component curve}, $M_{\mathcal{L}} \otimes L$ is strongly unstable for any line bundle $L$ on C, where $M_{\mathcal{L}}$ is the kernel bundle corresponding to the generated pair $(\mathcal{L},H^0(\mathcal{L}))$.
\end{example}

\section*{Acknowledgments}
We would like to express our gratitude to the anonymous referee for valuable comments and helpful remarks.
We thank Arijit Dey for suggesting this problem. We would like to thank  Arijit Dey, D. S. Nagaraj \& J. N. Iyer for many helpful discussions and comments. The first named author would like to thank National Board for Higher Mathematics (NBHM), Department of Atomic Energy, Government of India, for financial support through Postdoctoral Fellowship. The second named author would like to thank Council of Scientific \& Industrial Research (CSIR), India, for providing the financial assistance through CSIR - JRF,SRF scheme at the Department of Mathematics, Indian Institute of Technology, Madras. He would also like to thank Indian Institute of Technology, Madras and St. Joseph's College, Autonomous, Bangalore, for providing the necessary working environment.

\end{document}